\documentclass[12pt,a4paper]{amsart}
\usepackage{latexsym}
\calclayout

\let\le\undefined
\DeclareMathSymbol{\le}{\mathrel}{AMSa}{"36}         %\leqslant
\let\ge\undefined
\DeclareMathSymbol{\ge}{\mathrel}{AMSa}{"3E}         %\geqslant
\let\empty\undefined
\DeclareMathSymbol{\empty}{\mathord}{AMSb}{"3F}      %\varnothing

\renewcommand{\:}{\colon}
\renewcommand{\.}{\mskip 0.5\thinmuskip}
\newcommand{\rarrow}{\longrightarrow}
\newcommand{\lrarrow}{\.\relbar\joinrel\relbar\joinrel\rightarrow\.}
\newcommand{\llarrow}{\.\leftarrow\joinrel\relbar\joinrel\relbar\.}
\newcommand{\dsb}{\dotsb}
\newcommand{\dsc}{\dotsc}
\newcommand{\ot}{\otimes}

\newcommand{\B}{{\mathcal B}}
\newcommand{\C}{{\mathcal C}}
\newcommand{\F}{{\mathcal F}}
\newcommand{\E}{{\mathcal E}}
\newcommand{\Z}{{\mathcal Z}}
\renewcommand{\H}{{\mathcal H}}
\newcommand{\T}{{\mathcal T}}
\newcommand{\R}{{\mathbb R}}

\DeclareMathOperator{\Tor}{Tor}
\DeclareMathOperator{\Id}{Id}

\theoremstyle{plain}
\newtheorem*{thm}{Theorem}
\newtheorem*{cor}{Corollary}
\newtheorem*{lem}{Lemma}

\begin{document}

\title{The algebra of closed forms in a disk is Koszul}
\author{Leonid Positselski}

\address{Faculty of Mathematics and Laboratory of Algebraic Geometry,
National Research\- University Higher School of Economics,
Moscow 117312, and Sector of Algebra and Number Theory, Institute
for Information Transmission Problems, Moscow 127994, Russia}
\email{posic@mccme.ru}

\begin{abstract}
 We prove that the algebra of closed differential forms in
an (algebraic, formal, or analytic) disk with logarithmic
singularities along several coordinate hyperplanes is
(both nontopologically and topologically) Koszul.
 The connection with variations of mixed Hodge--Tate structures
is discussed in the introduction.
\end{abstract}


\maketitle

\vspace*{-\medskipamount}
\section*{Introduction}

 In this paper we consider the algebras of closed differential forms
in a disk, regular outside of several chosen coordinate hyperplanes
and having at most logarithmic singularities along these hyperplanes,
with respect to the operation of product of differential forms.
 Such algebras occur in connection with mixed Hodge--Tate sheaves
on smooth algebraic varieties~\cite{Lev}.
 More precisely, the above algebras of closed forms in a disk
play a role in the local description of such sheaves in
the neighborhood of a point that may belong either to the original
variety, or to a normal crossing divisor lying at infinity in its
smooth compactification.

 Let $D$ be a complex analytic disk and $V$ be the complement to
several coordinate hyperplanes in~$D$.
 According to~\cite{Lev}, the real mixed Hodge--Tate sheaves on $V$
with admissible singularities in $D\setminus V$ can be described in
terms of an associative, supercommutative, positively internally
graded DG-algebra $\R\H\T_{(V,D)}$.
 The cohomology of the DG-algebra $\R\H\T_{(V,D)}$ lie in the union
of two half-lines: the diagonal where the internal grading is equal
to the cohomological one and the axis where the cohomological grading
is equal to one.
 The diagonal part of the cohomology is isomorphic to the algebra of
closed forms in $V$ with logarithmic singularities along
$D\setminus V$, while the part lying in the latter axis (which is
responsible for the mixed Hodge structures over a point) has
a one-dimensional component in every positive internal degree.

 The commutative Hopf algebra describing the category of mixed
Hodge--Tate sheaves on $(V,D)$ in the Tannakian formalism is
the algebra of zero cohomology of the reduced bar-construction of
the DG-algebra $\R\H\T_{(V,D)}$.
 It follows from the Koszul property of the algebra of closed forms,
proven in this paper, that this bar-construction has no cohomology
in the cohomological gradings different from zero.
 To deduce this, it suffices to consider the spectral sequence
converging from the cohomology of the bar-construction of
the cohomology algebra of $\R\H\T_{(V,D)}$ to the cohomology of
the bar-construction of $\R\H\T_{(V,D)}$ itself, and use the well-known
description of the Ext algebra of the connected direct sum of
augmented algebras~\cite[Proposition~1.1 of Chapter~3]{PP}.
 The commutative Hopf algebra of zero cohomology is the cofree product
of the Hopf algebra quadratic dual to the algebra of closed
differential forms and the cofree Hopf algebra with homogeneous
cogenerators indexed by the positive integers.

 Various Koszul properties of the algebras of motivic cohomology
(Milnor K\+theory, Galois cohomology, etc.)\ play an important role in
the theory of motives (see~\cite{PV,Pbogom} and the other present
author's papers on the subject; the arguments above can be viewed as
a new illustration to this general observation).
 However, the present status of these properties is mostly that
of conjectures rather than theorems.
 The algebras of closed differential forms in a disk, which are
considered in this paper, provide an interesting family of algebras
of motivic significance whose Koszulity can be readily established.
 That will be demonstrated below.

 The author is grateful to Andrey Levin for posing the problem.
 This work was mostly done when both A.~L. and I were visiting
Max-Planck-Institut f\"ur Mathematik in Bonn in the Spring of 2003,
and I wish to thank the Institute for its hospitality.
 The author was supported by a grant from P.~Deligne 2004 Balzan prize,
a Simons Foundation grant, and RFBR grants while finalizing
the arguments and writing the paper up.

\section{Module Koszulity} \label{module-koszul}

 Let $D$ be a disk with the coordinates $z_1$,~\dots, $z_u$.
 With few exceptions, it will not matter for us which particular
geometric category is presumed.
 So $D$ can be the algebraic affine space over a field of
characteristic zero, the formal disk over such a field, a complex
analytic disk, or a smooth real disk.
 One can also take $D$ to be the spectrum of the algebra of polynomials
or formal power series with divided powers over a field of prime
characteristic.

 Let $0\le v\le u$.
 For any $1\le s\le v$, let $L_s$ denote the coordinate hyperplane
$\{z_s=0\}\subset D$.
 Denote by $\Omega$ the de~Rham DG\+algebra of regular differential
forms in $D\setminus\bigcup_{s=1}^v L_s$ with logarithmic singularities
along~$L_s$.
 Let $Z\subset\Omega$ denote the subalgebra of closed forms, i.~e.,
the kernel of the de~Rham differential $d\:\Omega\rarrow\Omega$.
 Let $H=H(\Omega,d)$ be the cohomology algebra of~$\Omega$.

 Denote by $A$ the exterior algebra generated by the closed $1$\+forms
$dz_s/z_s$, \ $s\le v$, and $dz_r$, \ $r>v$.
 It is only important for us to have control over the homological
properties of the $A$\+modules $\Omega$ and~$H$.
 In particular, one can replace the disk with any space endowed with
an \'etale map to the disk and satisfying an appropriate version of
the Poincar\'e lemma.
 In the above examples, $\Omega$ is the free $A$\+module generated
by $\Omega^0$, and $H$ is the exterior algebra generated by $dz_s/z_s$
with the obvious $A$\+module structure in which $dz_s/z_s\in A$ act
freely and $dz_r\in A$ act trivially in~$H$.

 The notion of a \emph{Koszul algebra}, introduced by
S.~Priddy~\cite{Pr} in the context of locally finite-dimensional
algebras with respect to an additional grading and studied mostly
for algebras with a finite-dimensional space of generators~\cite{PP},
can be easily generalized to the completely infinite-dimensional
case~\cite{PV}.
 The same applies~\cite{Pbogom} to the notion of a \emph{Koszul module}
introduced by A.~Beilinson, V.~Ginzburg, and W.~Soergel~\cite{BGS}.
 The main difference with the locally finite-dimensional case is
that in the infinite-dimensional situation the quadratic duality
connects algebras with coalgebras and modules with comodules.

 Let us recall these definitions.
 A nonnegatively graded algebra $A$ over a field~$k$ is called Koszul
if $A_0=k$ and $\Tor^A_{ij}(k,k)=0$ for $i\ne j$.
 For a Koszul algebra $A$, a nonnegatively graded left $A$\+module $M$
is called Koszul if $\Tor^A_{ij}(k,M)=0$ for $i\ne j$.
 Here the first index $i$ denotes the homological grading of the $\Tor$,
while the second index $j$ denotes the internal grading (see 
the references above).
 Note that for any nonnegatively graded algebra $A$ and module $M$
the condition $A_0=k$ implies the vanishing of $\Tor^A_{ij}(k,k)$ and
$\Tor^A_{ij}(k,M)$ for $i>j$.
 We will use the lower and upper indices interchangeably for denoting
our internal gradings; no sign change is presumed when passing from
the upper to the lower indices and back.

\begin{lem}
 Let $M$ be a nonnegatively graded left module over a Koszul
algebra~$A$.
 Then the graded vector space $\Tor_i^A(k,M)$ is concentrated in
the gradings $i$ and $i+1$ for all~$i$ if and only if
$M_+=M_1\oplus M_2\oplus\dsb$ is a Koszul left $A$-module in
the grading shifted by~$1$ (i.~e., so that the component $M_1$
be put in degree~$0$).
\end{lem}

\begin{proof}
 The assertion follows from the long exact sequence
$\dsb\rarrow\Tor^A_{i+1}(k,M_0)\rarrow\Tor^A_i(k,M_+)\rarrow
\Tor^A_i(k,M)\rarrow\Tor^A_i(k,M_0)\rarrow\dsb$,
since the graded vector space $\Tor^A_i(k,M_0)$ is concentrated
in degree~$i$ (due to Koszulity of the algebra~$A$).
\end{proof}

\begin{thm}
 Let $(\Omega,d)$ be a nonnegatively graded DG\+algebra over
a field~$k$ with the differential of degree~$1$; set $Z=\ker d$ and
$H=H(\Omega,d)$.
 Let $A$ be a Koszul algebra and $f\:A\rarrow Z$ be a morphism of
graded algebras.
 Assume that\/ $\Omega$ and $H$ are Koszul left $A$\+modules in
the module structures induced by~$f$.
 Then $Z^+ = Z^1\oplus Z^2\oplus\dsb$ is a Koszul left $A$\+module in
the grading shifted by~$1$.
\end{thm}

\begin{proof}
 For a graded $A$\+module $M$ and $j\in\mathbb Z$, denote by $M(j)$
the graded $A$\+module with the components $M(j)^m=M^{m-j}$ and
the action of $A$ defined by the rule $a\cdot x(j)=(-1)^{jn} (a\cdot
x)(j)$ for $x\in M$ and $a\in A^n$.
 Consider the complex of graded $A$\+modules
$$
\dsb\lrarrow \Omega(3)\lrarrow\Omega(2)\lrarrow\Omega(1)\lrarrow Z.
$$
 We are interested in the two hyperhomology spectral sequences with
the same limit that are obtained by applying the derived functor
$\Tor^A(k,{-})$ to this complex~$C$.
 Specifically, we have ${}'\!E^2_{pq}$, ${}''\!E^1_{pq}\implies
\Tor^A_{p+q}(k,C)$, where ${}'\!E^2_{pq}=\Tor^A_p(k,H(q))$ and
${}''\!E^1_{pq}=\Tor^A_q(k,\Omega(p))$ for $p>0$, while
${}''\!E^1_{0,q}=\Tor^A_q(k,Z)$.

 By the assumption, the term ${}'\!E^2_{pq}$ is concentrated in
the internal degree~$p+q$, hence the limit term $\Tor^A_i(k,C)$
is concentrated in the internal degree~$i$.
 Furthermore, the term ${}''\!E^1_{pq}$ is concentrated in
the internal degree~$p+q$ for all $p>0$.
 Now any components of the term ${}''\!E^1_{0,i}$ of the internal
degree different from~$i$ have to be killed by the differentials
${}''\!d^r\:{}''\!E^r_{r,i-r+1}\rarrow{}''\!E^r_{0,i}$, hence
the term ${}''\!E^1_{0,i}$ is concentrated in the internal
degrees $i$ and $i+1$.
 It remains to use Lemma.
\end{proof}

 For a nonnegatively graded $k$\+algebra $B$, set
$B'=k\oplus B_1\oplus B_2\oplus\dsb$.
 According to Lemma and~\cite[Theorem~6.1]{Pbogom}, if $A\rarrow B'$
is a homomorphism of graded algebras, the algebra $A$ is Koszul, and
$B_+$ is a Koszul left $A$\+module in the grading shifted by~$1$,
then the algebra $B'$ is Koszul.
 (For another proof of the same result, see Theorem in
Section~\ref{module-quasi} below.)
 Thus in the assumptions of Theorem all the three graded algebras
$\Omega'$, \ $H'$, and $Z'$ are Koszul.

\begin{cor}
 In any of the geometric categories listed above, the algebra $Z$
of closed differential forms in the disk $D$ with logarithmic
singularities along the several coordinate hyperplanes $L_s$
is Koszul.
\end{cor}

\begin{proof}
 It is clear from the discussion at the beginning of the section
that the algebra $\Omega$ of logarithmic differential forms in $D$
with respect to $\{L_s\}$ and its algebra of de~Rham cohomology
$H$ are Koszul modules over the exterior algebra $A$ generated
by $dz_s/z_s$ and $dz_r$.
 Besides, $Z^0=k$.
 Thus the assertion follows from Theorem above
and Theorem~6.1 from~\cite{Pbogom}.
\end{proof}

\section{Remarks on Topological Koszulity}  \label{quasi-topological}

 Koszulity of a graded algebra $Z$ is an exactness property of
the bar-complex of $Z$, whose components are direct sums of
tensor products of the grading components of~$Z$.
 However, tensor products of the grading components of an algebra
$Z$ are not always a natural object to consider.
 It is perfectly natural to consider such tensor products when $Z$
is the algebra of functions or forms on an algebraic variety, but
perhaps not so when $Z$ is the algebra of forms on a formal, complex
analytic, or smooth real disk.
 In the latter cases, it might be better to consider completed
tensor products.
 Notice in particular that it is the algebra of closed forms in
a complex analytic disk that appears in the problem of local
description of mixed Hodge--Tate sheaves (see Introduction to
the present paper and the preprint~\cite{Lev}).

 Specifically, one may wish to define the completed tensor product
$Z_{n_1}\widehat\otimes\dsb\widehat\otimes Z_{n_m}$ as the space
of closed forms in $D^m$ of degree~$n_t$ with respect to the $t$\+th
group of variables, with logarithmic singularities along
the hyperplanes $D^t\times L_s\times D^{m-t-1}$.
 Then one may construct the completed bar-complex out of such
completed tensor products and ask oneself whether it is exact
outside of the diagonal.
 Moreover, one may wish to consider the diagonal homology of this
complex as the completed version of the coalgebra Koszul dual to~$Z$.
 The component of degree~$n$ of this completed coalgebra is
the space of all closed forms in $D^n$ of degree~$1$ with respect
to each group of variables with vanishing pull-backs under the maps
$\Id_D^{t-1}\times \Delta_D\times\Id_D^{n-t-1}\:
D^{n-1}\rarrow D^n$, where $\Delta_D\:D\rarrow D^2$
is the diagonal map.

\medskip
 From the point of view of a homological algebraist, topological
algebra as such is a treacherous ground which is better avoided
whenever possible.
 So we propose a simple linear algebra formalism for describing
topological Koszulity in the above sense, independent on any notion
of a topological tensor product.

 With any associative algebra $Z$ with unit over a field~$k$, one can
associate the family of vector spaces $Z_n=Z^{\ot n}$, \ $n\ge0$,
endowed with linear maps induced by the multiplication and unit in~$Z$.
 The structure one obtains in this way is that of a simplicial
$k$\+vector space $Z_\bullet$ endowed with a morphism $k\rarrow
Z_\bullet$ into it from the constant simplicial vector space~$k$.

 When $Z$ is a graded algebra, one obtains a much richer structure.
 A (\emph{quasi-associative graded}) \emph{quasi-algebra} over
a field~$k$ is a family of vector spaces $\Z_{n_1,\dsc,n_m}$, where
$m\ge0$ and $n_t\in\mathbb Z$, endowed with
the \emph{quasi-multiplication} maps
$$
 \Z_{n_1,\dsc,n_m}\lrarrow
 \Z_{n_1,\dsc,n_{t-1},n_t+n_{t+1},n_{t+2},\dsc,n_m} 
$$
and the \emph{quasi-unit} maps
$$
 \Z_{n_1,\dsc,n_m}\lrarrow \Z_{n_1,\dsc,n_t,0,n_{t+1},\dsc,n_m}
$$
satisfying the conventional properties of the multiplication and unit
maps between the tensor products $\Z_{n_1,\dsc,n_m}=
Z_{n_1}\ot_k\dsb\ot_k Z_{n_m}$ of the grading components of
an associative algebra with unit.
 Specifically, the compositions of maps
$$
 \Z_{n_1,\dsc,n_m}\lrarrow
 \Z_{n_1,\dsc,n_s+n_{s+1},\dsc,n_t,\dsc,n_m} \lrarrow
 \Z_{n_1,\dsc,n_s+n_{s+1},\dsc,n_t+n_{t+1},\dsc,n_m}
$$
and
$$
 \Z_{n_1,\dsc,n_m}\lrarrow
 \Z_{n_1,\dsc,n_s,\dsc,n_t+n_{t+1},\dsc,n_m}\lrarrow 
 \Z_{n_1,\dsc,n_s+n_{s+1},\dsc,n_t+n_{t+1},\dsc,n_m}
$$
should coincide; the compositions of maps
$$
 \Z_{n_1,\dsc,n_m}\lrarrow
 \Z_{n_1,\dsc,n_{t-1}+n_t,\dsc,n_m} \lrarrow
 \Z_{n_1,\dsc,n_{t-1}+n_t+n_{t+1},\dsc,n_m}
$$
and
$$
 \Z_{n_1,\dsc,n_m}\lrarrow
 \Z_{n_1,\dsc,n_t+n_{t+1},\dsc,n_m} \lrarrow
 \Z_{n_1,\dsc,n_{t-1}+n_t+n_{t+1},\dsc,n_m}
$$
should coincide.
 The quasi-unit maps must commute with the quasi-multiplication
maps; the compositions
$$
 \Z_{n_1,\dsc,n_m}\lrarrow\Z_{n_1,\dsc,n_t,0,n_{t+1},\dsc,n_m}
 \lrarrow\Z_{n_1,\dsc,n_t+0,n_{t+1},\dsc,n_m} = \Z_{n_1,\dsc,n_m}
$$
and
$$
 \Z_{n_1,\dsc,n_m}\lrarrow\Z_{n_1,\dsc,n_t,0,n_{t+1},\dsc,n_m}
 \lrarrow\Z_{n_1,\dsc,n_t,0+n_{t+1},\dsc,n_m} = \Z_{n_1,\dsc,n_m}
$$
should be the identity maps.
 The component $\Z_\empty$ with $m=0$ indices has to be identified
with~$k$.
 (Cf.\ \cite[Section~4 of Chapter~3]{PP}.)

 A quasi-algebra is said to be \emph{nonnegative} if
$\Z_{n_1,\dsc,n_m}=0$ whenever $n_t<0$ for some $1\le t\le m$.
 A nonnegative quasi-algebra is said to be \emph{positive} if
all its quasi-unit maps are isomorphisms.
 A positive quasi-algebra is determined by its components
$\Z_{n_1,\dsc,n_m}$ with positive indices $n_m>0$ and
the quasi-multiplication maps between them subject to
the quasi-associativity equations (i.~e., those of the above
equations which depend on the quasi-multiplication maps only).

 Let $Z$ be a positively graded associative algebra over~$k$, i.~e.,
$Z$ is nonnegatively graded in the obvious sense and $Z_0=k$.
 Set $Z_+=Z/k$.
 Then one can associate with $Z$ its reduced bar-complex $B$ of
the form
 $$
  k\llarrow Z_+\llarrow Z_+\ot_k Z_+\llarrow Z_+\ot_k Z_+\ot_k Z_+
  \llarrow\dsb,
 $$
consider its grading component $B_n$ of degree~$n$, and take
the tensor product $\B_{n_1,\dsc,n_m}=B_{n_1}\ot_k\dsb\ot_k B_{n_m}$
of several such complexes.
 The components of the complex $\B_{n_1,\dsc,n_m}$ are direct sums
of tensor products of the grading components of the algebra~$Z$.
 Replacing all such tensor products with the components
$\Z_{n'_1,\dsc,n'_{m'}}$ of an arbitrary positive quasi-algebra, one
defines the complex $\B_{n_1,\dsc,n_m}$ as an additive functor on
the abelian category of positive quasi-algebras.
 A positive quasi-algebra $\Z$ is called \emph{Koszul} if all
the complexes $\B_{n_1,\dsc,n_m}$ have no homology except at
the homological degree $n_1+\dsb+n_m$.

 Inverting all the arrows in the above definitions, one defines
(\emph{quasi-coassociative graded}) \emph{quasi-coalgebras} and,
in particular, \emph{positive quasi-coalgebras} and \emph{Koszul
quasi-coalgebras}.
 In particular, the quasi-comultiplications are the maps
$$
 \C_{n_1,\dsc,n_{t-1},n_t+n_{t+1},n_{t+2},\dsc,n_m}\lrarrow
 \C_{n_1,\dsc,n_m},
$$
the quasi-counits are the maps
$$
 \C_{n_1,\dsc,n_t,0,n_{t+1},\dsc,n_m}\lrarrow\C_{n_1,\dsc,n_m},
$$
and Koszulity of positive quasi-coalgebras is defined in
terms of the complexes emulating tensor products of the grading
components of reduced cobar-complexes of positively graded coalgebras.

 The additive categories of Koszul quasi-algebras and Koszul
quasi-coalgebras are equivalent.
 The equivalence functor assigns to a Koszul quasi-algebra $\Z$
the Koszul quasi-coalgebra $\C$ with the components
$\C_{n_1,\dsc,n_m}=H_{n_1+\dsb+n_m}(\B_{n_1,\dsc,n_m})$.
 The natural surjections of complexes
$$
 \B_{n_1,\dsc,n_{t-1},n_t+n_{t+1},n_{t+2},\dsc,n_m}\lrarrow
 \B_{n_1,\dsc,n_m}
$$
induce the quasi-comultiplications in~$\C$.
 The inverse functor is defined in the similar way in terms of
the cobar-complexes of quasi-coalgebras.

 In particular, all the quasi-multiplication maps between components
with nonnegative indices in a Koszul quasi-algebra are surjective.
 This is the quasi-algebra analogue of the condition that the graded
algebra $Z$ be generated by~$Z_1$.
 There is also an analogue of the quadraticity condition, and
there are the dual conditions for Koszul quasi-coalgebras.
 Due to these conditions, the dual Koszul quasi-algebra $\Z$ and
quasi-coalgebra $\C$ are uniquely determined by the vector spaces
$\Z_{1,\dsc,1}\simeq\C_{1,\dsc,1}$ and the exact sequences
$$
 0\lrarrow\C_{1,\dsc,1,2,1,\dsc,1}\lrarrow\C_{1,\dsc,1}\.\simeq
 \.\Z_{1,\dsc,1}\lrarrow\Z_{1,\dsc,1,2,1,\dsc,1}\lrarrow0.
$$
 For the corresponding quasi-coalgebra $\C$ and quasi-algebra $\Z$
to be Koszul, the collection of $n-1$ subspaces $\C_{1,\dsc,1,2,1,
\dsc,1}$ in the vector space $\C_{1,\dsc,1}$ ($n$~units) has to
be distributive for all~$n$ (cf.~\cite[Subsection~2.2]{PV}).

 A quasi-algebra with \emph{external multiplications} is
a quasi-algebra $\Z$ endowed with linear maps
$$
 \Z_{n_1,\dsc,n_t}\ot_k \Z_{n_{t+1},\dsc,n_m}\lrarrow
 \Z_{n_1,\dsc,n_m}
$$
compatible with the identification $\Z_\empty=k$ and commuting
with the quasi-multiplication and quasi-unit maps.
 Quasi-coalgebras with external multiplications are defined in
the similar way.
 This time, arrows are not inverted, i.~e., external
multiplications in a quasi-coalgebra have the form
$$
 \C_{n_1,\dsc,n_t}\ot_k \C_{n_{t+1},\dsc,n_m}\lrarrow
 \C_{n_1,\dsc,n_m}.
$$
 The quasi-algebra or quasi-coalgebra corresponding to a graded
algebra or coalgebra is naturally endowed with external
multiplications; and a quasi-(co)algebra with external
multiplications comes from a (uniquely defined) graded (co)algebra
if and only if all its external multiplication maps are isomorphisms.

 The equivalence between the categories of Koszul quasi-algebras
and Koszul quasi-coalgebras transforms quasi-algebras with
external multiplications to quasi-coalgebras with external
multiplications and back.
 In order to see this, it suffices to construct natural external
multiplications
$$
 \B_{n_1,\dsc,n_t}\ot_k \B_{n_{t+1},\dsc,n_m}\lrarrow
 \B_{n_1,\dsc,n_m}
$$
on the bar-complexes of a quasi-algebra with external
multiplications, and similar external multiplications on
the cobar-complexes of a quasi-coalgebra with external
multiplications.

\section{Module Koszulity for Quasi-Algebras}  \label{module-quasi}

 Let $A$ be a Koszul algebra and $\Z$ be a positive quasi-algebra
over a field~$k$.
 Assume that $A$ acts in $\Z$ from the left in the following sense:
for all $l_p$,~\dots, $l_1$, $n_0$,~\dots, $n_q\ge0$, $\,p$, $q\ge0$
there are linear maps
$$
 A_n\otimes_k\Z_{l_p,\dsc,l_1,n_0,\dsc,n_q}\lrarrow \Z_{l_p,\dsc,l_1,
 n+n_0,n_1,\dsc,n_q},
$$
making $\bigoplus_{n_0=0}^\infty \Z_{l_p,\dsc,l_1,n_0,\dsc,n_q}$
a graded $A$\+module for any fixed
$l_p$,~\dots, $l_1$, $n_1$,~\dots, $n_q$ and commuting with
the quasi-multiplication maps
$$
 \Z_{l_p,\dsc,l_1,n_0,\dsc,n_q}\lrarrow\Z_{l_p,\dsc,l_1,n_0,\dsc,n_t+n_{t+1},
\dsc,n_q}
$$
for all $0\le t<q$.

 For any positively graded associative algebra $Z$ over~$k$,
consider the tensor product $E_{l_p,\dsc,l_0;n_1,\dsc,
n_q} = Z_{l_p}\ot_k\dsb\ot_k Z_{l_0}\ot_k B_{n_1}\ot_k\dsb\ot_k B_{n_q}$
of the grading components of the algebra $Z$ and its reduced
bar-complex~$B$ (with the trivial coefficients).
 The terms of the complex $E_{l_p,\dsc,l_0;n_1,\dsc,n_q}$ are direct sums
of tensor products of grading components of the algebra $Z$;
replacing all such tensor products with the components
$\Z_{n_1',\dsc,n'_{m'}}$ of a positive quasi-algebra $\Z$, we obtain
the complex $\E_{l_p,\dsc,l_0;n_1,\dsc,n_q}$.
 The action of $A$ in the quasi-algebra $\Z$ induces its action
$$
 A_n\otimes_k\E_{l_p,\dsc,l_0;n_1,\dsc,n_q}\lrarrow \E_{l_p,\dsc,l_1,
 n+l_0;n_1,\dsc,n_q}
$$
on the complexes $\E_{l_p,\dsc,l_0;n_1,\dsc,n_q}$, hence also on their
cohomology.

\begin{thm}
 Suppose that the graded $A$\+module
$$\textstyle
 \bigoplus_{l_0=1}^\infty H_{n_1+\dsb+n_q}(\E_{l_p,\dsc,l_0;n_1\dsc,n_q})
$$
is Koszul in the grading shifted by~$1$ for any fixed\/
$l_p$,~\dots, $l_1$, $n_1$,~\dots, $n_q\ge1$, $\,p$,  $q\ge0$.
 Then the quasi-algebra $\Z$ is Koszul.
\end{thm}

\begin{proof}
 Let us show that the homology of the complex
$\E_{l_p,\dsc,l_1;n_0,\dsc,n_q}$
are concentrated in the degree $n_0+\dsb+n_q$ for all $l_p$,~\dots,
$l_1$, $n_0$,~\dots, $n_q\ge1$, $\,p$, $q\ge0$.
 We will argue by induction in $n_0+\dsb+n_q$.

 Imagine that we have a morphism of positively graded algebras
$A\rarrow Z$ and consider the bicomplex
$F=\bigoplus_{t',t''=0}^\infty A_+^{\ot t'}\ot_k Z\ot_k Z_+^{\ot t''}$
with one differential induced by the differential of the reduced
bar-complex of $A$ with coefficients in the left $A$\+module $Z$
and the other one induced by the differential of the reduced
bar-complex of $Z$ with coefficients in the right $Z$\+module~$Z$.
 Consider the component $F_{n_0}$ of the complex $F$ of the (total)
internal degree~$n_0$ and take its tensor product
$Z_{l_p}\ot_k\dsb\ot_k Z_{l_1}\ot_kF_{n_0}\ot_k B_{n_1}\ot_k
\dsb\ot_k B_{n_q}$ with the components of the algebra $Z$ and its
reduced bar-complex $B$ (with the trivial coeffcients).
 Denote the complex so obtained by $F_{l_p,\dsc,l_1;n_0;n_1,\dsc,n_q}$.
 The components of this complex are direct sums of tensor products
of the grading components of the algebras $A$ and $Z$.
 Moving the tensor products of components of the algebra $A$
to the left in these tensor products and replacing all the tensor
products of components of an algebra $Z$ with the components
$\Z_{n'_1,\dsc,n'_{m'}}$ of a positive quasi-algebra, we obtain
the complex $\F_{l_p,\dsc,l_1;n_0;n_1,\dsc,n_q}$ depending as an additive
functor on the quasi-algebra $\Z$ with a left action of~$A$.

 Using the quasi-algebra version of the contracting homotopy of
the bar-complex of $Z$ with coefficients in $Z$, which is induced
by the isomorphisms $\Z_{l_p,\dsc,l_1,n'_0,\dsc,n'_{q'}}\rarrow
\Z_{l_p,\dsc,l_1,0,n'_0,\dsc,n'_{q'}}$, one can see that
the complex $\F_{l_p,\dsc,l_1;n_0;n_1,\dsc,n_q}$ is quasi-isomorphic to
the tensor product of the component of internal degree~$n_0$ of
the reduced bar-complex of $A$ (with the trivial coefficients)
and the complex $\E_{l_p,\dsc,l_1;n_1,\dsc,n_q}$.
 So it follows from the induction hypothesis that the homology of
$\F_{l_p,\dsc,l_1;n_0;n_1,\dsc,n_q}$ is concentrated in the homological
degree $n_0+\dsb+n_q$.
 On the other hand, there is a natural surjective morphism of
complexes $\F_{l_p,\dsc,l_1;n_0;n_1,\dsc,n_q}\rarrow
\E_{l_p,\dsc,l_1;n_0,\dsc,n_q}$.
 We will show that the homology of the kernel of this map is
concentrated in the homological degrees $n_0+\dsb+n_q$ and
$n_0+\dsb+n_q-1$; so it will follow that the homology of
$\E_{l_p,\dsc,l_1;n_0,\dsc,n_q}$ is concentrated in the degree
$n_0+\dsb+n_q$.

 Indeed, let $\F'_{l_p,\dsc,l_1;n_0;n_1,\dsc,n_q}$ be the subcomplex of
$\F_{l_p,\dsc,l_1;n_0;n_1,\dsc,n_q}$ consisting of all the components in
which all the indices of the tensor factor
$\Z_{l_p,\dsc,l_1,n'_0,\dsc,n'_{q'}}$ are positive (i.~e., $n'_0\ge1$).
 Then the quotient complex of $\F_{l_p,\dsc,l_1;n_0;n_1,\dsc,n_q}$ by
$\F'_{l_p,\dsc,l_1;n_0;n_1,\dsc,n_q}$ is isomorphic to the direct sum
of the tensor products of the component of internal grading~$n$
of the reduced bar-complex of $A$ (with the trivial coeffcients)
and the complex $\E_{l_p,\dsc,l_1;n_0-n,n_1,\dsc,n_q}$.
 It follows from the induction assumption that the kernel of the map
$$
 \F_{l_p,\dsc,l_1;n_0;n_1,\dsc,n_q}/\F'_{l_p,\dsc,l_1;n_0;n_1,\dsc,n_q}
 \lrarrow \E_{l_p,\dsc,l_1;n_0,\dsc,n_q}
$$
has homology in the homological grading $n_0+\dsb+n_q$ only.

 On the other hand, the complex $\F'_{l_p,\dsc,l_1;n_0;n_1,\dsc,n_q}$ has
a finite increasing filtration whose indices are the sums
$n'_1+\dsb +n'_{q'}$ of indices of the tensor factors
$\Z_{l_p,\dsc,l_1,n'_0,\dsc,n'_{q'}}$. 
 The associated quotients of this filtration are the components of
internal grading~$n$ of the bar-complexes of $A$ with coefficients
in the complexes of graded $A$\+modules
$\bigoplus_{l_0=1}^\infty \E_{l_p,\dsc,l_1,l_0;n_0-n,n_1,\dsc,n_q}$.
 It follows from the induction assumption that they are quasi-isomorphic
to the components of internal grading~$n$ of the bar-complexes of $A$
with coefficients in the graded modules
$$\textstyle
 \bigoplus_{l_0=1}^\infty H_{n_0+\dsb+n_q-n}
 (\E_{l_p,\dsc,l_1,l_0;n_0-n,n_1,\dsc,n_q}),
$$
placed in the homological grading $n_0+\dsb+n_q-n$.
 By the module Koszulity assumption of Theorem, their homology are
concentrated in the homological degree $n_0+\dsb+n_q-1$.
(Cf.~\cite[proof of Theorem~6.1]{Pbogom}.)
\end{proof}

 Let us finally return to the case when $\Z_{n_1,\dsc,n_m}=
Z_{n_1}\widehat\otimes\dsb\widehat\otimes Z_{n_m}$ is
the quasi-algebra of closed forms in a formal, complex analytic, or
smooth real disk, as defined in Section~\ref{quasi-topological}.
 One defines the quasi-multiplication in $\Z$ in terms of
the diagonal embedding $\Delta_D\: D\rarrow D\times D$.
 The vector space $H_{n_1+\dsb+n_q}(\E_{l_p,\dsc,l_0;n_1\dsc,n_q})$
is described as the space of closed forms in $D^{p+1+n_1+\dsb+n_q}$
of the degrees $l_1$,~\dots, $l_0$ with respect to the first $p+1$
groups of variables and $1$ with respect to the last
$n_1+\dsb+n_q$ groups, regular outside of the hyperplanes
$D^t\times L_s\times D^{p+1+n_1+\dsb+n_q-t}$ and having at most
logarithmic singularities along these hyperplanes,
with vanishing inverse images with respect to the maps
$$
 \Id_D^{p+n_1+\dsb+n_{t-1}+s}\times\Delta_D\times
 \Id_D^{n_t-s-1+n_{t+1}+\dsb+n_q}\:D^{p+n_1+\dsb+n_q}\rarrow
 D^{p+1+n_1+\dsb+n_q}
$$
for all $1\le s\le n_t-1$.

 To show that such spaces of closed forms satisfy the module
Koszulity assumption of Theorem above, it suffices to consider
them as spaces of closed forms of the degree~$l_0$ in the disk $D$
with coordinates from the (p+1)\+th group of variables, depending
on the variables from all the other groups as parameters.
 One has to include such parameters into the assertion of Theorem
from Section~\ref{module-koszul} and use the Poincar\'e lemma
with parameters (provable by the usual integration procedure,
which preserves the conditions imposed on the dependence of our
forms from parameters).

 Hence $\Z$ is Koszul, and its quadratic dual quasi-coalgebra $\C$
described in Section~\ref{quasi-topological} is Koszul, too.
 Besides, $\Z$ and $\C$ are obviously a quasi-algebra and
a quasi-coalgebra with external multiplications, and their
external multiplication structures correspond to each other
under the Koszul duality.

\end{document}